\documentclass[10pt,reqno]{amsart}
\usepackage{amsmath,amsfonts,mathscinet,amssymb,amsthm}
\usepackage{enumitem,xcolor}
\usepackage[mathcal]{euscript}
\usepackage[numbers,sort&compress]{natbib}
\usepackage[a5paper, left=8mm, right=8mm, top=1.5cm, bottom =1cm]{geometry}
\usepackage[colorlinks, linkcolor=teal, citecolor=magenta]{hyperref}

\setlist[enumerate]{label={\rm(\roman*)}}

\newcommand{\dd}{{\fam0 d}}
\newcommand{\dom}{\mathcal{R}}
\let\embed\hookrightarrow
\newcommand{\sgn}{\operatorname{sgn}}
\let\tilde\widetilde

\newtheorem{theorem}{Theorem}[section]
\newtheorem*{theorem*}{Theorem}
\newtheorem{lemma}[theorem]{Lemma}

\newtheorem{remark}[theorem]{Remark}

\numberwithin{equation}{section}
\expandafter\let\expandafter\oldproof\csname\string\proof\endcsname
\let\oldendproof\endproof
\renewenvironment{proof}[1][\proofname]{%
  \oldproof[\bf #1]%
}{\oldendproof}

\begin{document}

\keywords{interpolation, supremum operator, Marcinkiewicz space, Lorentz gamma space}
\subjclass[msc2010]{46E30, 46B70}%

\title[Interpolation theorem for Marcinkiewicz spaces]{Interpolation theorem for Marcinkiewicz spaces with applications to Lorentz gamma spaces}

\author[V\'\i t Musil]{V\'\i t Musil\textsuperscript{1}}

\author[Rastislav O{\soft{l}}hava]{Rastislav O{\soft{l}}hava\textsuperscript{1,2}}

\email[V.~Musil (Corresponding author)]{musil@karlin.mff.cuni.cz}
\urladdr{0000-0001-6083-227X}
\email[R.~O\soft{l}hava]{olhava@karlin.mff.cuni.cz}
\urladdr{0000-0002-9930-0454} 

\address{\textsuperscript{1}%
Department of Mathematical Analysis,
Faculty of Mathematics and Physics, 
Charles University,
So\-ko\-lo\-vsk\'a~83,
186~75 Praha~8,
Czech Republic}  

\address{\textsuperscript{2}%
Institute of Applied Mathematics
and Information Technologies,
Faculty of Science, 
Charles University,
Al\-ber\-tov 6, 
128~43 Praha~2,
Czech Republic}

\begin{abstract}
  This paper is devoted to the interpolation principle between spaces of weak type. 
We characterise interpolation spaces between two Marcinkiewicz spaces in terms of Hardy type operators involving suprema. We study general properties of such operators and their behavior on Lorentz gamma spaces.
A~particular emphasis is placed on elementary and comprehensive proofs.
\end{abstract}

\bibliographystyle{abbrv}

\maketitle     

\section*{How to cite this paper}
\noindent
Musil V., O\soft{l}hava R. Interpolation theorem for Marcinkiewicz spaces with applications to Lorentz gamma spaces. \emph{Mathematische Nachrichten}, 292(5):1106-1121, 2019.
\begin{center}
	\url{https://doi.org/10.1002/mana.201700452}
\end{center}

\section{Introduction}

Let $\dom=(\dom,\mu)$ be non-atomic $\sigma$-finite measure space with $\mu(\dom)=R$,
where $0<R\le\infty$. Let $\mathcal{M}(\dom,\mu)$ denote the~collection of all
extended real-valued $\mu$-measurable and a.e.~finite functions on~$\dom$.

This paper deals with Marcinkiewicz interpolation theorem
between spaces of weak type where the norm is defined by
\begin{equation*}  
	\|f\|_{M_\varphi(\dom)} = \sup_{0<s<R} \varphi(s) f^{**}(s) .
\end{equation*} 
Here $\varphi$ is so-called quasiconcave function (for the definition see
Section~\ref{sec:back}), the double stars stand for the maximal function
defined as a Hardy average of $f^*$,
\begin{equation*}
	f^{**}(t) = \frac{1}{t} \int_{0}^{t} f^{*}(s)\,\dd s,
\end{equation*}
in which $f^*$ represents the non-increasing rearrangement of $f$, given by
\begin{equation*}
	f^*(t)=\inf
		\bigl\{
			\lambda > 0;\, \mu
			(\{
				x\in \dom;\, |f(x)|>\lambda
			\})
			\leq t
		\bigr\},
	\quad t\in[0,R).
\end{equation*}
The collection of all functions
$f\in\mathcal{M}(\dom,\mu)$ with $\|f\|_{M_\varphi(\dom)}$ finite
is called Marcinkiewicz space $M_\varphi(\dom,\mu)$.

In our main result we prove that the boundedness of a certain operator is ensured by that of the supremum operators or, more precisely, Hardy-type operators involving suprema  
$S_\varphi$ and $T_\psi$ defined by
\begin{equation*}
	S_{\varphi} f(t) = \frac {1} {\varphi(t)} \sup_{0<s<t} \varphi(s) f^*(s), 
	\quad t\in(0,R),\quad f\in \mathcal{M}(\dom,\mu),
\end{equation*}
\begin{equation*}
	T_{\psi} f(t) =  \frac {1} {\psi(t)} \sup_{t<s<R} \psi(s) f^*(s),
	\quad t\in(0,R),\quad f\in \mathcal{M}(\dom,\mu),
\end{equation*}
where $\varphi$ and $\psi$ are quasiconcave functions. Such a result was first proved by Dmitriev and Kre\u\i n
in \cite{Dmi:78}; however, the supremum operators appeared only implicitly.
Later, Kerman and Pick in \cite{Ker:06} and \cite{Ker:09}  showed the equivalence of the boundedness of  the operators of such kind and certain Sobolev-type embeddings and they also used their result in the search of optimal pairs of rearrangement invariant spaces for which these embeddings hold.
Consequently, Kerman, Phipps and Pick in \cite{Ker:14} found simple criteria for the boundedness of the supremum operators on Orlicz spaces and Lorentz Gamma spaces and they obtained corresponding Marcinkiewicz interpolation theorems.  However, all the above-mentioned results concern only power functions in place
of~$\varphi$. In this work, we want to fill this gap.

The principal innovation of this paper consists not only in a significant extension of the known results but also in the new and more elegant comprehensive approach that enables us to establish proofs which are more enlightening and illustrative and less technical than those applied in earlier works.
More specifically, the authors in \cite{Ker:06} used the
discretization method to show some of the properties of the supremum operators
involved
and the main result in \cite{Ker:14} is achieved by the method using K-functionals and Holmstedt formulas. Instead, our way is rather elementary, using the
basic properties of rearrangements.

We will work in the general setting of rearrangement-invariant (r.i.\ for
short) Banach function spaces $X(\dom, \mu)$ as collections of all
$\mu$-measurable functions finite a.e.~on $\dom$ such that
$\|f\|_{X(\dom,\mu)}$ is finite.

One can define an r.i.~space $X(\dom,\mu)$ on a general measure space
$(\dom,\mu)$ using rearrangement invariance of the given r.i.~space $X(0,R)$,
\begin{equation*}
	\|f\|_{X(\dom,\mu)} =\|f^*\|_{X(0,R)},
	\quad f\in\mathcal{M}(\dom,\mu).
\end{equation*}
On the other hand, there is also a representation of each norm of a given
r.i.~space $X(\dom,\mu)$ by r.i.~norm 
on interval due to the Luxemburg representation theorem.  For further information
regarding r.i.~norms see \cite[Chapter~1 and 2]{Ben:88}.
At the places where no confusion is likely to happen, we shall use a shorter
form $X(\dom)$ instead of $X(\dom,\mu)$.

We also exhibit the general properties of the supremum operators $S_\varphi$
and $T_\psi$ like the endpoint embeddings in the r.i.~class
(Section~\ref{sec:endpoints}) or the relation to the maximal function
(Section~\ref{sec:stars}). It turns out that a certain averaging condition on
the quasiconcave function plays a key role here. It
reads as
\begin{equation*}
	\frac{1}{t} \int_0^t \frac{\dd s}{\varphi(s)}
		\lesssim \frac{1}{\varphi(t)},
			\quad t\in(0,R).
\end{equation*}
We shall refer to this relation as a $B$-condition and write $\varphi\in B$.
More details about quasiconcave functions and $B$-condition can be found in
Section~\ref{sec:back}.

Our principal result now reads as follows.

\begin{theorem}
Let $\dom_1=(\dom_1,\mu_1)$ and $\dom_2 = (\dom_2,\mu_2)$ be non-atomic $\sigma$-finite measure spaces for which $\mu_1(\dom_1) = \mu_2(\dom_2)=R$. Suppose that a quasilinear operator $T$ satisfies
\begin{equation*}
	T\colon M_{\varphi}(\dom_1) \to M_{\varphi}(\dom_2)
	\quad\text{and}\quad
	T\colon M_{\psi}(\dom_1) \to M_{\psi}(\dom_2)
\end{equation*}
for quasiconcave functions $\varphi,\psi$ defined on $[0,R)$, both satisfying the B-condition and let $X_i(\dom_i)$, $i=1,2$, be r.i.~spaces satisfying 
\begin{equation*}
	M_{\varphi}(\dom_i)\cap M_{\psi}(\dom_i)
		\subset X_i(\dom_i) \subset
	M_{\varphi}(\dom_i) + M_{\psi}(\dom_i),
		\quad i=1,2.
\end{equation*}
Then
\begin{equation*}
	T\colon X_1 (\dom_1) \to X_2 (\dom_2)
\end{equation*}
whenever
\begin{equation}
	S_{\varphi}\colon X_1(0,R) \to X_2(0,R)
		\quad\text{and}\quad
	T_{\psi}\colon X_1(0,R) \to X_2(0,R).
	\label{eq:DKP}
\end{equation}
\label{thm:DKP}
\end{theorem}

\begin{remark} \label{rem:WhichOne}
The role of the quasiconcave functions $\varphi$ and $\psi$ in Theorem~\ref{thm:DKP} is actually interchangeable since the conditions on the operator $T$ are symmetric. Thus, the conditions
\begin{equation*} \tag{\ref{eq:DKP}*}
	S_{\psi}\colon X_1(0,R) \to X_2(0,R)
		\quad\text{and}\quad
	T_{\varphi}\colon X_1(0,R) \to X_2(0,R)
\end{equation*}
are also sufficient for $T\colon X_1 (\dom_1) \to X_2 (\dom_2)$ and one can
consider to use them instead of \eqref{eq:DKP}, if convenient. In the
case when both $S_\varphi f\lesssim S_\psi f$ and $T_\psi f\lesssim  T_\varphi f$ are
satisfied for every $f\in X_1(0,R)$, clearly {\rm(\ref{eq:DKP}*)} implies
\eqref{eq:DKP}, so the condition \eqref{eq:DKP} is weaker and most likely
easier to verify for a reader. We characterise when this occurs
in~Lemma~\ref{lemm:WhichOne}.

\end{remark}

Our next result concerns the criteria to guarantee \eqref{eq:DKP} in specific class of r.i.\ spaces, namely in the classical Lorentz gamma spaces $\Gamma^p_{w}(\dom)$ where the norm is given as
\begin{equation*}
	\|f\|_{\Gamma^p_{w}(\dom)} = \biggl( \int_{0}^{R} [ f^{**}(t) ]^p {w}(t)\,\dd t \biggr)^\frac{1}{p}.
\end{equation*}
Here $1\le p<\infty$ and ${w}$ is some positive and locally integrable function, so-called weight.
We require $\int_{1}^{\infty} s^{-p}{w}(s)\,\dd s <\infty$ when $R=\infty$ and $\int_{0}^{R} s^{-p}{w}(s)\,\dd s = \infty$
when $R<\infty$ otherwise $\Gamma^p_{w} = \{0\}$ in the first case and $\Gamma^p_w = L^1$ in the second one. Such
requirements are called nontriviality conditions.

If we deal with the operator $S_\varphi$ acting between Lorentz gamma spaces with discontinuous $\varphi$, we moreover admit additional nontriviality conditions, i.e, we assume
\begin{equation}
	\int_0^R \varphi^{-p}(s){w}(s) \,\dd s < \infty
			\label{nontriv}
\end{equation}
and
\begin{equation}
		\lim_{t\to 0^+} t^p \int_{t}^{R} s^{-p} {w}(s)\,\dd s > 0.
			\label{Emb-g-ln}
\end{equation}
As we explain in Remark~\ref{rem:discontinuity}, such requirements are necessary and cause no loss of generality.

\begin{theorem}
Let $\dom_1 = (\dom_1,\mu_1)$ and $\dom_2 = (\dom_2,\mu_2)$ be non-atomic $\sigma$-finite measure spaces with $\mu_1(\dom_1)=\mu_2(\dom_2)=R$, $\varphi$ and $\psi$ be quasiconcave functions defined on $[0,R)$ satisfying the B-condition, ${w}_1$ and ${w}_2$ be nontrivial weights on $(0,R)$. In the case $\varphi$ is not continuous, let, in addition, ${w}_1$ and ${w}_2$ satisfy \eqref{Emb-g-ln} and \eqref{nontriv} respectively.
Let~$p$ be an index, $1\leq p<\infty$, such that
\begin{equation*}
M_{\varphi}(\dom_i)\cap M_{\psi}(\dom_i) \subset \Gamma^p_{{w}_i}(\dom_i) \subset  M_{\varphi}(\dom_i)+ M_{\psi}(\dom_i),
	\quad i=1,2.
\end{equation*}
Suppose $T$ is a quasilinear operator that satisfies
\begin{equation*}
	T \colon M_{\varphi}(\dom_1) \to M_{\varphi}(\dom_2)
		\quad \text{and} \quad
	T \colon M_{\psi}(\dom_1) \to M_{\psi}(\dom_2);
\end{equation*}
then, a sufficient condition for the embedding
\begin{equation*}
	T \colon \Gamma^p_{{w}_1}(\dom_1) \to \Gamma^p_{{w}_2}(\dom_2)
\end{equation*}
is as follows
\begin{equation}
	\sup_{0<t<R}\frac{
	\psi^p(t) \int_{0}^{t} \psi^{-p}(s){w}_2(s)\,\dd s
	+ \varphi^p(t) \int_{t}^{R} \varphi^{-p}(s){w}_2(s)\,\dd s
	}{
	\int_0^t {w}_1(s)\,\dd s
	+
	t^p \int_t^R s^{-p}{w}_1(s)\,\dd s
	} < \infty.
	\label{STggpp}
\end{equation}
\label{thm:A}
\end{theorem}

The proof of this result follows from a characterisation of the boundedness of the supremum operators $S_\varphi$ and $T_\varphi$ between two Lorentz gamma spaces, of independent interest, formulated in the following two theorems.

\begin{theorem}
Let $1\le p <\infty$ and let $\dom = (\dom,\mu)$ be non-atomic $\sigma$-finite measure spaces with $\mu(\dom)=R$. Suppose that $\varphi$ is quasiconcave function on $[0,R)$ satisfying the $B$-condition and assume that ${w}_1$, ${w}_2$ are nontrivial weights on $(0,R)$. In the case $\varphi$ is not continuous, let, in addition, ${w}_1$ and ${w}_2$ satisfy \eqref{Emb-g-ln} and \eqref{nontriv} respectively.
Then
	\begin{equation}
		S_\varphi\colon \Gamma^p_{{w}_1}(\dom) \to \Gamma^p_{{w}_2}(0,R)
		\label{Sgg}
	\end{equation}
	holds if and only if
	\begin{equation}
		\sup_{0<t<R}
		\frac{\int_0^t {w}_2(s)\,\dd s
			+	\varphi^p(t) \int_t^R \varphi^{-p}(s) {w}_2(s) \,\dd s}
		{\int_0^t {w}_1(s)\,\dd s
		+
		t^p \int_t^R s^{-p}{w}_1(s)\,\dd s}
		< \infty.
			\label{Sggpp}
	\end{equation}
\label{thm:Sgg}
\end{theorem}

\begin{theorem}
Let $1\le p <\infty$ and let $\dom = (\dom,\mu)$ be non-atomic $\sigma$-finite measure spaces with $\mu(\dom)=R$. Suppose that $\psi$ is a quasiconcave function on $[0,R)$ satisfying the $B$-condition and let ${w}_1$, ${w}_2$ be nontrivial weights on $(0,R)$. Then
	\begin{equation}
		T_{\psi}\colon \Gamma^p_{{w}_1}(\dom) \to \Gamma^p_{{w}_2}(0,R)
		\label{Tgg}
	\end{equation}
	holds if and only if
	\begin{equation}
		\sup_{0<t<R}
		\frac{\psi^p(t) \int_0^t {\psi}^{-p}(s) {w}_2(s) \,\dd s
			+ t^p \int_t^R s^{-p}{w}_2(s)\,\dd s}
		{\int_0^t {w}_1(s)\,\dd s
			+ t^p \int_t^R s^{-p}{w}_1(s)\,\dd s}
		< \infty.
			\label{Tggpp}
	\end{equation}
\label{thm:Tgg}
\end{theorem}

The proofs of these results appear in Section~\ref{sec:proofs}. 
One may also notice that we sometimes avoid stating the results in the full generality. For instance, one may try to extend Theorems~\ref{thm:Sgg} and \ref{thm:Tgg} to various exponents on the left and the right hand sides or avoid the $B$-condition. The reason is similar here; the general situation can be treated by discretization methods while we want to keep the approach as elementary as possible.

\section{Quasiconcave functions} \label{sec:back}

Let us recall that if a non-negative function defined on $[0,R)$, $\varphi$, satisfies
\begin{enumerate}
	\item $\varphi(t)=0$ if and only if $t=0$;
	\item $\varphi$ is non-decreasing;
	\item $\varphi(t)/t$ is non-increasing on $(0,R)$,
\end{enumerate}
then $\varphi$ is said to be \textit{quasiconcave}.
If we denote $\tilde{\varphi}(t) = t/\varphi(t)$ for $t\in(0,R)$  and $\tilde{\varphi}(0)=0$ then $\tilde{\varphi}$ is also a quasiconcave function. We say that $\tilde{\varphi}$ is complementary function to $\varphi$.

A quasiconcave function $\varphi$ is continuous in every positive argument from
its domain. Any jump at such a point would lead to the contradiction with the
monotonicity of complementary function $\tilde{\varphi}$ or $\varphi$ itself.
Only possible point of discontinuity of quasiconcave functions is zero.
Even stronger, each quasiconcave function is absolutely continuous on every closed subinterval of $(0,R)$, which enables us to write it as the integral of its derivative for any $a<b$; $a,b\in [0,R)$,
\begin{equation}
\label{eq:acc}
\varphi(b)-\varphi(a{\scriptstyle +})=\int_a^b \varphi'(s)\,\dd s, 
\end{equation} 
where $\varphi(a{\scriptstyle +})=\lim_{t\to a^+} \varphi(t)$. The only case when $\varphi(a{\scriptstyle +})$ does not equal to $\varphi(a)$ is for discontinuous $\varphi$ at zero. For the reference see \cite[Chapter II, Lemma~1.1]{Kre:82}.

As we already mentioned before, many important properties of operators
or spaces based on the quasiconcave functions rely on the so-called
$B$-condition (the notation $B_1$-condition is also occurred in the
literature).
There are many other applications of this simple condition and thereby
lot of equivalent properties are known. Among others, let
us mention that $\varphi\in B$ if and only if
there exists a constant $k>1$ such that
\begin{equation*}
	\inf_{0<t<R} \frac{\widetilde{\varphi}(kt)}{{\widetilde\varphi}(t)} > 1.
\end{equation*}
Based on this criterion, one can easily decide whether $\varphi\in B$ or not.
For further details, one can look at
\cite[Lemma~2.3]{Str:79},
\cite[Chapter II, Lemma~1.4]{Kre:82} or
\cite[Theorem~12]{Kuf:07}.

Let us mention the first easy simplification of the Marcinkiewicz norm.
Note that this was already observed by several authors, see for instance
\cite[Theorem~5.3]{Kre:82} or
\cite[Theorem~9.5]{ONe:68}.
However, for the sake of self-contained reading we include a simple proof.
Here and in the sequel we use the notation $A\lesssim B$ if $A\le CB$
where $C$ is a constant independent of all quantities obtained in $A$ and $B$.
In the case $A\lesssim B$ and $B\lesssim A$ we will use $A\simeq B$.

\begin{lemma} \label{lemm:marc_one_star}
Let $\varphi$ be a quasiconcave function on $[0,R)$. Then
\begin{equation} \label{eq:mos}
	\sup_{0<t<R} \varphi(t) f^{**}(t)
		\simeq \sup_{0<t<R} \varphi(t) f^{*}(t)
\end{equation}
for every $f\in\mathcal{M}(\dom,\mu)$ if and only if $\varphi\in B$.
\end{lemma}

Before we provide the proof of this lemma, let us state another auxiliary
result which will be of use in the sequel.

\begin{lemma} \label{lemm:MeasurableExist}
	Let $\dom=(\dom,\mu)$ be non-atomic $\sigma$-finite measure space with $\mu(\dom)=R$ and let $\xi$ be a non-increasing nonnegative function on $(0,R)$. Then there exists a $\mu$-measurable $f$ defined on $\dom$ such that $f^*=\xi$.
\end{lemma}

\begin{proof}
Let $\{\xi_n\}_{n=1}^\infty$ be a sequence of simple function defined on $(0,R)$ such that $\xi_n\uparrow\xi$. We can assume that $\xi_n$ are non-increasing, since $\xi_n\uparrow\xi$ implies $\xi_n^*\uparrow\xi^*=\xi$.
Let us write
\begin{equation*}
	\xi_n=\sum_{j=1}^{m_n} b_{j}^{n}\chi_{[0,t_{j}^{n})},
\end{equation*}
where $0<t_{j}^{n}<t_{k}^{n}\le R$ for any $1\le k<j\le m_n$.  Now,
\cite[Chapter 2, Lemma~2.5]{Ben:88} applied on $\chi_{\dom}$ guarantees the
existence of $\mu$-measurable sets $E_{j}^{n}$ such that
$\mu(E_{j}^{n})=t_{j}^{n}$ and $E_{j}^{n}\subset E_{k}^{\tilde{n}}$ whenever
$t_{j}^{n}<t_{k}^{\tilde{n}}$, for any $1\le j\le m_n$, $1\le k\le
m_{\tilde{n}}$. Then, the sequence $\{f_n\}_{n=1}^\infty$ defined by
\begin{equation*}
	f_n=\sum_{j=1}^{m_n} b_{j}^{n}\chi_{E_{j}^{n}},
\end{equation*}
consists of $\mu$-measurable simple functions satisfying $f_n^*=\xi_n$ and the sequence 
$\{f_n\}$ increases monotonically. Using the properties of the non-increasing
rearrangement, the point-wise limit of $f_n$, say $f$,  is $\mu$-measurable and
$f^*=\xi$. 
\end{proof}

\begin{proof}[Proof of Lemma~\ref{lemm:marc_one_star}]
The necessity follows immediately by setting $f\in\mathcal{M}(\mu,\dom)$ with $f^*=1/\varphi$. The existence of such $f$ is ensured by Lemma~\ref{lemm:MeasurableExist}. The condition \eqref{eq:mos} tells us that 
\begin{equation*}
\sup_{0<t<R} \varphi(t) \left(\frac{1}{\varphi}\right)^{**}(t)=\sup_{0<t<R} \frac{\varphi(t)}{t}\int_0^t \frac{\dd s}{\varphi(s)}
\end{equation*}
is bounded, what is exactly $\varphi\in B$.

As for the sufficiency, suppose that $\varphi\in B$. Since $f^{*}\le f^{**}$ the left hand side of \eqref{eq:mos} dominates the right hand side of \eqref{eq:mos}. For the opposite inequality denote the right hand side of \eqref{eq:mos} by $M$. We then have
\begin{equation*}
	f^*(s) \le M \frac{1}{\varphi(s)},
		\quad s\in (0,R).
\end{equation*}
Integrating this inequality over $(0,s)$ and dividing by $s$ we get
\begin{equation*}
	f^{**}(t)
		\le \frac{M}{t} \int_0^t \frac{\dd s}{\varphi(s)}
		\lesssim M \frac{1}{\varphi(t)},
		\quad t\in (0,R),
\end{equation*}
hence
\begin{equation*}
	\sup_{0<t<R} \varphi(t)f^{**}(t) \lesssim M
\end{equation*}
as we wished to show.
\end{proof}

\begin{remark}
Note that for a given $\mu$-measurable function $f$, both $T_\psi f$ and $S_\varphi f$ are non-increasing functions. Indeed,
\begin{align*}
	S_\varphi f (t)
		& = \frac{ 1 }{\varphi(t) } \sup_{0<s<t} \varphi(s) \sup_{s<y<R} f^*(y)
			\\
		& = \frac{ 1 }{\varphi(t) } \sup_{0<y<R} f^{*}(y) \sup_{0<s<\min \{ t,y \} } \varphi(s)
			\\
		& = \sup_{0<y<R} f^{*}(y) \min\biggl\{ 1, \frac{\varphi(y)}{\varphi(t)}\biggr\}
\end{align*}
which is clearly non-increasing. The case concerning $T_\psi f$ is obvious.
\end{remark}

We conclude this section by characterisation of the situation mentioned in Remark~\ref{rem:WhichOne}.

\begin{lemma} \label{lemm:WhichOne}
Let $\varphi$ and $\psi$ be quasiconcave functions on $[0,R)$. Then
\begin{enumerate}
\item 
	\begin{equation*}
 		S_\varphi f\lesssim S_\psi f,
            	\quad f\in \mathcal{M}(\dom),
            	\quad\text{if and only if}\
        \sup_{0<s<t}\frac{\varphi(s)}{\psi(s)}\lesssim\frac{\varphi(t)}{\psi(t)},
            \quad t\in(0,R);
	\end{equation*}
\item 
	\begin{equation*}
    	T_\psi g\lesssim  T_\varphi g,
        	\quad  g\in \mathcal{M}(\dom),
                \quad\text{if and only if}\ 
      	\sup_{t<s<R}\frac{\psi(s)}{\varphi(s)}\lesssim\frac{\psi(t)}{\varphi(t)},
			\quad t\in(0,R).
	\end{equation*}
\end{enumerate}
Moreover, both supremal conditions are equivalent.
\end{lemma}

\begin{proof}
The necessity of the supremal conditions follows immediately by testing the
particular inequality by $\mu$-measurable function $f$ such that $f^*=1/\psi$
in (i), and $g$ such that $g^*=1/\varphi$ in (ii). The existence of such
functions is guaranteed by Lemma~\ref{lemm:MeasurableExist}.
We have
\begin{equation*}
\frac {1} {\varphi(t)} \sup_{0<s<t} \frac{\varphi(s)}{\psi(s)} =S_\varphi f\lesssim S_\psi f=\frac {1} {\psi(t)},\quad t\in(0,R)
\end{equation*}
and
\begin{equation*}
\frac {1} {\psi(t)} \sup_{t<s<R} \frac{\psi(s)}{\varphi(s)} =T_\psi g\lesssim T_\varphi g=\frac {1} {\varphi(t)},\quad t\in(0,R). 
\end{equation*}

As for the sufficiency, using the supremal condition, one can show directly that
\begin{align*}
S_\varphi f(t)	&=\frac {1} {\psi(t)}\frac {\psi(t)} {\varphi(t)} \sup_{0<s<t} \frac{\varphi(s)}{\psi(s)} \psi(s)f^*(s) \\ 
				&\le\left(\frac {\psi(t)} {\varphi(t)}\sup_{0<y<t} \frac{\varphi(y)}{\psi(y)}\right)\frac {1} {\psi(t)} \sup_{0<s<t} \psi(s)f^*(s)\lesssim S_\psi f(t),\quad f\in\mathcal{M}(\dom)
\end{align*}
and analogously
\begin{align*}
T_\psi g(t)	&=\frac {1} {\varphi(t)}\frac {\varphi(t)} {\psi(t)} \sup_{t<s<R} \frac{\psi(s)}{\varphi(s)} \varphi(s)g^*(s) \\ 
				&\le\left(\frac {\varphi(t)} {\psi(t)}\sup_{t<y<R} \frac{\psi(y)}{\varphi(y)}\right)\frac {1} {\varphi(t)} \sup_{t<s<R} \varphi(s)g^*(s)\lesssim T_\varphi g(t),\quad g\in\mathcal{M}(\dom).
\end{align*}

Now, the supremal condition in (i) is equivalent to the assertion that
$\varphi/\psi$ is equivalent to some non-decreasing function. On the other
hand, the supremal condition in (ii) is equivalent to the claim that 
$\psi/\varphi$ is equivalent to a non-increasing  function. These statements
mean the same, obviously.  
\end{proof}

\section{Endpoint estimates} \label{sec:endpoints}

In this section we will focus on endpoint mapping properties for supremum
operators in the setting of  r.i.~spaces. These kinds of estimates belong among
general properties of operators and  represent the  basic tool to work with
such operators. We are especially interested in the boundedness of supremum operators on
Marcinkiewicz spaces, since we use it several times in the main
proofs. 

The first result of this kind, the endpoint estimates for supremum operators
with power function, can be found in \cite[Lemma~3.5]{Ker:06}. It turns out
that in our more general case for supremum operators with a quasiconcave
function the B-condition comes into scene. The
sufficiency of the B-condition was presented in \cite[Chapter~3]{Olh:11};
however, the proofs were rather technical and the necessity was not considered
at all. In the sequel we will present more elegant approach and show the
equivalency.

\begingroup
\let\varphi\psi
	\begin{lemma} \label{lemm:endpoint_T}
		Let $\varphi$ be a quasiconcave function on $[0,R)$. Then
		\begin{enumerate}
			\item
			\begin{equation*}
				T_\varphi \colon L^1(\dom) \to L^1(\dom)
					\quad \text{if and only if}\quad
				\varphi \in B;
			\end{equation*}
			\item
			\begin{equation*}
				T_\varphi \colon M_\varphi(\dom) \to M_\varphi(\dom)
					\quad \text{if and only if}\quad
				\varphi \in B.
			\end{equation*}
		\end{enumerate}
	\end{lemma}

	\begin{proof}
		For the necessity of the $B$-condition in (i) we just apply $T_\varphi$ on $f=\chi_{E_t}$, where $E_t$ is $\mu$-measurable with $\mu(E_t)=t$ for any $t\in(0,R)$. Such $f$ can be found by Lemma~\ref{lemm:MeasurableExist}.  Clearly, $f^*=\chi_{(0,t)}$ and we can just evaluate the norms
        \begin{equation*}
        \bigl\|T_\varphi \chi_{E_t}\bigr\|_{L^1(\dom)}
					\lesssim \| \chi_{E_t}\|_{L^1(\dom)}
        \end{equation*}
				obtaining the $B$-condition
        \begin{equation*}
        \varphi(t)\int_0^t\frac{\dd s}{\varphi(s)}\lesssim t,
				 \quad t\in(0,R),
        \end{equation*}
        as needed.
        Testing by the same function gives us the necessity also in (ii).
				Indeed, since
      \begin{equation*}
			T_\varphi \chi_{E_t}(s)
				= \chi_{(0,t)}(s) \frac{\varphi(t)}{\varphi (s)},
		\end{equation*}
        we have 
        \begin{equation*}
			\bigl( T_\varphi \chi_{E_t}\bigr)^{**} (s)
				=	\chi_{(0,t)}(s) \frac{\varphi(t)}{s}\int_0^s \frac{\dd y}{\varphi(y)}
					+	\chi_{(t,R)}(s) \frac{\varphi(t)}{s}\int_0^t \frac{\dd y}{\varphi(y)}
		\end{equation*}
        for every pair $s$ and $t$ in $(0,R)$. Now, clearly $\| \chi_{E_t}\|_{M_\varphi(\dom)}=\varphi(t)$ and the boundedness of $T_\varphi$ on $M_\varphi(\dom)$ yields 
        \begin{equation*}
        	\chi_{(0,t)}(s) \frac{\varphi(s)}{s}\int_0^s \frac{\dd y}{\varphi(y)}
					+	\chi_{(t,R)}(s) \frac{\varphi(s)}{s}\int_0^t \frac{\dd y}{\varphi(y)}
						\lesssim 1,
						\quad s\in(0,R),
						\quad t\in(0,R).
        \end{equation*}
        In situation $s<t$, we obtain $\varphi\in B$.
        
		For the sufficiency in (i), we split the integration in two parts, namely
		\begin{align*}
			\bigl\|T_\varphi f\bigr\|_{L^1(\dom)}
				& = \int_0^R \frac{1}{\varphi(t)} \sup_{t<s<R} \varphi(s) f^{*}(s)\,\dd t
					\\
				& \le	\int_0^R \frac{1}{\varphi(t)} \sup_{t<s<R} \bigl( \varphi(s) - \varphi(t) \bigr) f^{*}(s)\,\dd t
					+ \int_0^R \sup_{t<s<R} f^{*}(s)\,\dd t
					\\
				& = \text{I} + \text{II}.
		\end{align*}
		The second part equals to the $L^1(\dom)$ norm of $f$, while the first part needs some estimates. We use \eqref{eq:acc} and
		\begin{align*}
			\text{I}
				& =	\int_0^R \frac{1}{\varphi(t)}\sup_{t<s<R} \Bigl( \int_t^s \varphi'(y)\,\dd y \Bigr) f^*(s)\,\dd t
					\\
				& \le \int_0^R \frac{1}{\varphi(t)}\sup_{t<s<R} \Bigl( \int_t^s \varphi'(y) f^*(y)\,\dd y \Bigr) \dd t
					\\
				& = \int_0^R \frac{1}{\varphi(t)} \int_t^R \varphi'(y)f^{*}(y)\,\dd y\,\dd t
					\\
				& = \int_0^R \varphi'(y)f^{*}(y) \int_0^y \frac{\dd t}{\varphi(t)} \,\dd y
					 \\
				& \lesssim \int_0^R \frac{y}{\varphi(y)}\varphi'(y)f^{*}(y)\,\dd y
					 \\
				& \le \int_0^R f^{*}(y)\,\dd y,
					\\
		\end{align*}
		where the first inequality follows from the monotonicity of $f^*$. The
		third term equals the fourth one by the Fubini theorem. Then we use
		$B$-condition for $\psi$ and for the last one, we use the quasiconcavity of
		$\psi$.
		This completes the proof of the part (i).

		For~the~sufficiency in the part (ii), recall that $T_\varphi f$ is
		non-increasing and hence we have
		\begin{align*}
			\|T_\varphi f \|_{M_\varphi(\dom)}
				& = \sup_{0<t<R} \varphi(t) \bigl(T_\varphi f \bigr)^{**}(t)
					\\
				& = \sup_{0<t<R} \frac{\varphi(t)}{t} \int_0^t \frac{1}{\varphi(s)} \sup_{s<y<R}\varphi (y)f^*(y)\,\dd s
					\\
				& \le \sup_{0<t<R} \frac{\varphi(t)}{t} \int_0^t \frac{1}{\varphi(s)} \sup_{0<y<R}\varphi (y)f^{**}(y)\,\dd s
					\\
				& = \| f \|_{M_\varphi(\dom)} \sup_{0<t<R} \frac{\varphi(t)}{t}\int_0^t \frac{\dd s}{\varphi(s)}.
		\end{align*}
		The last supremum is finite because of the~$B$-condition for $\varphi$.
	\end{proof}
\endgroup

	\begin{lemma} \label{lemm:endpoint_S}
		Let $\varphi$ be a quasiconcave function on $[0,R)$. Then
		\begin{enumerate}
			\item
			\begin{equation*}
				S_\varphi \colon M_\varphi(\dom) \to M_\varphi(\dom)
					\quad \text{if and only if}\quad
				\varphi \in B;
			\end{equation*}
			\item
			\begin{equation*}
				S_\varphi \colon L^\infty(\dom) \to L^\infty(\dom)
					\quad \text{for every quasiconcave $\varphi$}.
			\end{equation*}
		\end{enumerate}
	\end{lemma}

	\begin{proof}
		Let us consider part (i). For the necessity we set $f=\chi_{E_a}$ with $f^*=\chi_{(0,a)}$, similarly as in Lemma~\ref{lemm:endpoint_T}.
		We obtain 
		\begin{equation*}
			S_\varphi \chi_{E_a} (t)
				= \min\biggl\{ 1, \frac{\varphi(a)}{\varphi(t)} \biggr\},
					\quad t\in (0,R),
                    \quad a\in (0,R),
		\end{equation*}
		and thus for every $a\in(0,R)$ we have
		\begin{align*}
			\| S_\varphi \chi_{E_a} \|_{M_\varphi(\dom)}
				& = \sup_{0<t<R} \frac{\varphi(t)}{t} \int_0^t \min\biggl\{ 1, \frac{\varphi(a)}{\varphi(s)} \biggr\}\dd s
					\\
				& = \sup_{0<t<R} \varphi(t) \chi_{(0,a)}(t) +
					\frac{\varphi(t)}{t}\biggl( a + \varphi(a) \int_a^t \frac{\dd s}{\varphi(s)} \biggr) \chi_{(a,R)}(t)
					\\
				& \ge \varphi(a)\sup_{a<t<R} \frac{\varphi(t)}{t} \int_a^t \frac{\dd s}{\varphi(s)}.
		\end{align*}
		Clearly $\|\chi_{E_a}\|_{M_\varphi(\dom)} = \varphi(a)$ and
		since $S_\varphi$ is bounded on $M_\varphi(\dom)$, we get
		\begin{equation*}
			\varphi(a)\sup_{a<t<R} \frac{\varphi(t)}{t} \int_a^t \frac{\dd s}{\varphi(s)}
				\le \varphi(a),
			\quad a\in(0,R).
		\end{equation*}
		The term $\varphi(a)$ cancels and, by taking the limit $a\to 0^+$, we get
		the $B$-condition.

		Now suppose that $\varphi\in B$. Taking Lemma~\ref{lemm:marc_one_star} and
		the monotonicity of $S_\varphi f$ into account we have
		\begin{align*}
			\| S_\varphi f \|_{M_\varphi(\dom)}
				& = \sup_{0<t<R} \varphi(t) \bigl( S_\varphi f \bigr)^{**}(t)
				\simeq \sup_{0<t<R} \varphi(t) \bigl( S_\varphi f \bigr)^{*}(t)
					\\
				& = \sup_{0<t<R} \varphi(t) \bigl( S_\varphi f \bigr)(t)
				= \sup_{0<t<R} \sup_{0<s<t} \varphi(s) f^*(s)
					\\
				& \simeq \sup_{0<s<R} \varphi(s) f^{**}(s)
					= \|f\|_{M_\varphi(\dom)}.
		\end{align*}

		Part (ii) is trivial.
	\end{proof}

\section{Starfalls} \label{sec:stars}

\begingroup
\let\varphi\psi
	\begin{lemma}
		Let $\varphi$ be a quasiconcave function on $[0,R)$. Then
		\begin{enumerate}
			\item
				\begin{equation*}
						\bigl( T_\varphi f \bigr)^{**}
							\lesssim T_\varphi f + f^{**}
                            \ \text{for every}\ f\in \mathcal{M}(\dom,\mu),
							\quad \text{if and only if}\quad
						\varphi \in B;
				\end{equation*}
			\item
				\begin{equation*}
						T_\varphi f^{**}
							\simeq T_\varphi f + f^{**}
                            \ \text{for every}\ f\in \mathcal{M}(\dom,\mu),
							\quad \text{if and only if}\quad
						\varphi \in B.
				\end{equation*}
		\end{enumerate}
		\label{lemm:prop_T}
	\end{lemma}

	\begin{proof}
	To prove the necessity in (i), we test the inequality again by $f=\chi_{E_a}$
	with $f^*=\chi_{(0,a)}$, as in Lemma~\ref{lemm:endpoint_T}.
	We have
		\begin{equation*}
			T_\varphi \chi_{E_a}(t)
				= \chi_{(0,a)}(t) \frac{\varphi(a)}{\varphi (t)}
		\end{equation*}
		and
		\begin{equation*}
			\bigl( T_\varphi \chi_{E_a}\bigr)^{**} (t)
				=	\chi_{(0,a)}(t) \frac{\varphi(a)}{t}\int_0^t \frac{\dd s}{\varphi(s)}
					+	\chi_{(a,R)}(t) \frac{\varphi(a)}{t}\int_0^a \frac{\dd s}{\varphi(s)}
		\end{equation*}
		for every pair $a$ and $t$ in $(0,R)$.
		The necessity of the $B$-condition then follows by comparing the appropriate quantities for arbitrary $t<a$.
        
        As for the necessity in the case (ii), we need to test the equivalence by $f_a\in\mathcal{M}(\mu,\dom)$ with $f_a^*(s)=\min\{1/\psi(a),1/\psi(s)\}$ for some $a\in(0,R)$, which exists according to Lemma~\ref{lemm:MeasurableExist}. Let us compute all the terms occurred in the condition. We have
        \begin{equation*}
       		f_a^{**}(t)=\chi_{(0,a)}(t) \frac{1}{\psi(a)}
					+	\chi_{(a,R)}(t) 
                    	\left( \frac{a}{t}\frac{1}{\psi(a)}
                    		+ \frac{1}{t}\int_a^t \frac{\dd y}{\psi(y)}\right)
        \end{equation*}
        and
        \begin{equation*}
       		T_\psi f_a^{**}(t)=\frac{1}{\psi(t)}\sup_{t<s<R}\left[\chi_{(0,a)}(s) \frac{\psi(s)}{\psi(a)}
					+	\chi_{(a,R)}(s) 
                    	\left( \frac{a}{s}\frac{\psi(s)}{\psi(a)}
                    		+ \frac{\psi(s)}{s}\int_a^s \frac{\dd y}{\psi(y)}\right)\right]
        \end{equation*}
        and also $T_\psi f_a(t)=1/\psi(t)$. In the situation when $t\ge a$, the condition implies 
        \begin{equation*}
        	\frac{1}{\psi(t)}\sup_{t<s<R} 
                    	\left( \frac{a}{s}\frac{\psi(s)}{\psi(a)}
                    		+ \frac{\psi(s)}{s}\int_a^s \frac{\dd y}{\psi(y)}\right)
        	\lesssim \frac{1}{\psi(t)} + \frac{a}{t}\frac{1}{\psi(a)}
                    		+ \frac{1}{t}\int_a^t \frac{\dd y}{\psi(y)}.
        \end{equation*}
        Now, multiplying by ${\psi(t)}$ and by the choice  $t=a$, we obtain
		\begin{equation*}
        \sup_{a<s<R} 
                    	\left( \frac{a}{s}\frac{\psi(s)}{\psi(a)}
                    		+ \frac{\psi(s)}{s}\int_a^s \frac{\dd y}{\psi(y)}\right)
        	\lesssim 1, 
        \end{equation*}
        which leads to the $B$-condition for $\psi$ by taking the limit $a\to 0^+$.
        
		To prove the sufficiency in (i), we divide the outer integral into three parts.
		\begin{align*}
			\frac1t \int_0^t 	\frac{1}{\varphi(y)} \sup_{y<s<R} \varphi(s) f^*(s)\,\dd y
				& \le	\frac1t \int_0^t \frac{1}{\varphi(y)} \sup_{y<s<t} \bigl( \varphi(s) - \varphi(y) \bigr) f^*(s)\,\dd y
					\\
					& \qquad + \frac1t \int_0^t \sup_{y<s<t} f^*(s)\,\dd y
					\\
					& \qquad + \frac1t \int_{0}^{t} \frac{1}{\varphi(y)}\sup_{t<s<R} \varphi(s) f^*(s)\,\dd y
					\\
				& \quad = \text{I} + \text{II} + \text{III}.
		\end{align*}
		The first term can be treated in the same way as in the proof of Lemma~\ref{lemm:endpoint_T}, part (i). We get $\text{I}\lesssim f^{**}(t)$.
		The term II clearly equals $f^{**}(t)$. Finally, since $\varphi\in B$,
		\begin{equation*}
			\text{III}
				 \lesssim \frac{1}{\varphi(t)} \sup_{t<s<R} \varphi(s)f^*(s)
					= T_\varphi f(t),
						\quad t\in(0,R).
		\end{equation*}
		Adding all these estimates together, we have
		\begin{equation*}
			\bigl( T_\varphi f \bigr)^{**}(t)
				\lesssim f^{**}(t) + T_\varphi f (t),
					\quad t\in (0,R).
		\end{equation*}
		
		Let us show the equivalence (ii) assuming $\varphi\in B$. One inequality is obvious since $f^{**} \le T_\varphi f^{**}$.
		The reversed inequality can be observed by the splitting argument
		similar to that in part (i). For $t\in(0,R)$, we have
		\begin{align*}
			T_{\varphi} f^{**} (t)
				& = \frac{1}{\varphi (t) }\sup_{t<s<R} \frac{\varphi(s)}{ s }\int_0^s f^*(y)\,\dd y
					\\
				& \le \frac{1}{\varphi (t) }\sup_{t<s<R} \frac{\varphi(s)}{s }\int_t^s f^*(y)\,\dd y
					+ \frac{1}{\varphi (t) }\sup_{t<s<R} \frac{\varphi(s)}{s }\int_0^t f^*(y)\,\dd y
					\\
				& = \text{I} + \text{II}.
		\end{align*}
		Surely $\text{II} = f^{**}(t)$
		and
		\begin{align*}
			\text{I}
				& = \frac{1}{\varphi (t) } \sup_{t<s<R} \frac{\varphi(s)}{s }\int_t^s \varphi(y)f^*(y)\frac{\dd y}{\varphi(y)}
					\\
				& \le \frac{1}{\varphi (t) } \sup_{t<y<R} \varphi(y)f^*(y) \sup_{t<s<R} \frac{\varphi(s)}{ s }\int_t^s \frac{\dd y}{\varphi(y)},
                \end{align*}
		which we get by taking the supremum of the part of the integrand out. The
		first part of the last term is exactly $T_{\varphi} f(t)$ and the second
		one can be estimated by taking $t=0$ and using the $B$-condition for
		$\varphi$ as follows
				\begin{align*}
				\frac{1}{\varphi (t) } \sup_{t<y<R} \varphi(y)f^*(y) \sup_{t<s<R} \frac{\varphi(s)}{ s }\int_t^s \frac{\dd y}{\varphi(y)}
                	& \le T_{\varphi} f (t) \sup_{0<s<R} \frac{\varphi(s)}{s }\int_0^s \frac{\dd y}{\varphi(y)}
						\\
					& \lesssim T_{\varphi} f (t).
					\qedhere
		\end{align*}
	\end{proof}
\endgroup

	\begin{lemma} \label{lemm:prop_S}
		Let $\varphi$ be a quasiconcave function on $[0,R)$. Then
		\begin{enumerate}
			\item
			\begin{equation*}
				\bigl( S_\varphi f \bigr)^{**}
					\lesssim S_\varphi f^{**}
                     \ \text{for every}\ f\in \mathcal{M}(\dom,\mu),
					\quad \text{if and only if}\quad
				\varphi \in B;
			\end{equation*}
			\item
			\begin{equation*}
				S_\varphi f^{**}
					\lesssim S_\varphi f
                     \ \text{for every}\ f\in \mathcal{M}(\dom,\mu),
					\quad \text{if and only if}\quad
				\varphi \in B.
			\end{equation*}
		\end{enumerate}
	\end{lemma}

	\begin{proof}
	Part (i). The necessity follows by plugging $f=\chi_{E_a}$
	with $f^*=\chi_{(0,a)}$ into the inequality.
		We have
		\begin{equation*}
			S_\varphi \chi_{E_a} (t)
				= S_\varphi \chi^{**}_{E_a} (t)
				= \min\biggl\{ 1, \frac{\varphi(a)}{\varphi(t)} \biggr\},
					\quad t\in (0,R),
					\quad a\in(0,R).
		\end{equation*}
		We calculate
		\begin{equation*}
			\bigl( S_\varphi \chi_{E_a} \bigr)^{**}(t)
				= \chi_{(0,a)}(t) + \frac{1}{t}\biggl( a + \varphi(a) \int_a^t \frac{\dd s}{\varphi(s)} \biggr) \chi_{(a,R)}(t),
					\quad a\in(0,R),
					\quad t\in(0,R),
		\end{equation*}
		hence for $t>a$ we have
		\begin{equation*}
			\bigl( S_\varphi \chi_{E_a} \bigr)^{**}(t)
				\ge \frac{\varphi(a)}{t}\int_a^t \frac{\dd s}{\varphi(s)},
		\end{equation*}
		therefore for those $t$ and $a$ we get
		\begin{equation*}
			\frac{\varphi(a)}{t}\int_a^t \frac{\dd s}{\varphi(s)}
				\lesssim \frac{\varphi(a)}{\varphi(t)}.
		\end{equation*}
		The term $\varphi(a)$ cancels and, by taking the limit $a\to 0^+$, we obtain the $B$-condition.

		On the other side, we have
		\begin{multline*}
			\varphi(t)\bigl( S_\varphi f\bigr)^{**}(t)
				\le \sup_{0<s<t} \varphi(s)\bigl( S_\varphi f\bigr)^{**}(s)
				\\
				= \bigl\| S_\varphi f\bigr\|_{M_\varphi(0,t)}
				\lesssim \|f\|_{M_\varphi(0,t)}
				= \sup_{0<s<t} \varphi(s) f^{**}(s)
		\end{multline*}
		thanks to Lemma~\ref{lemm:endpoint_S}. Dividing by $\varphi(t)$ we get the result.

		Part (ii) follows immediately with the help of Lemma~\ref{lemm:marc_one_star} by
		\begin{equation*}
			\varphi(t)\, S_\varphi f (t)
				= \sup_{0<s<t} \varphi(s) f^{*}(s)
				\simeq \sup_{0<s<t} \varphi(s) f^{**}(s)
				= \varphi(t)\, S_\varphi f^{**} (t),
					\quad t\in(0,R).
					\qedhere
		\end{equation*}
	\end{proof}

	\begin{lemma} \label{lemm:stars_out}
		Let $0<R\le\infty$ and let $\varphi$ and $\psi$ be quasiconcave functions on $[0,R)$. Then
		\begin{equation*}
			S_{\varphi} f^{**} + T_{\psi} f^{**}
				\simeq
			S_{\varphi} f + T_{\psi} f
		\end{equation*}
		for every $f\in \mathcal{M}(\dom,\mu)$ if and only if both $\varphi\in B$ and $\psi\in B$ hold.
	\end{lemma}

	\begin{proof}
	The claim is a corollary of Lemma~\ref{lemm:prop_T} since
	\begin{equation*}
		T_{\psi} f^{**}
			\lesssim T_{\psi} f + f^{**}
			\le T_{\psi} f + S_{\varphi} f^{**}
	\end{equation*}
	and Lemma~\ref{lemm:prop_S} which ensures that
	\begin{equation*}
		S_{\varphi} f^{**} \lesssim S_{\varphi} f.
	\end{equation*}
	The opposite inequality is obvious. As for the necessity, we need to use two test functions that eliminate one of the operators and provide $B$-condition for one of the quasiconcave functions involved. Let us use the test functions from the subsidiary lemmas in the appropriate way. Testing by $f_a\in\mathcal{M}(\dom, \mu)$ with $f_a^*(s)=\min\{1/\psi(a),1/\psi(s)\}$ for all $a\in(0,R)$ provides $\psi\in B$, since
    \begin{equation*}
    	S_{\varphi} f_a^{**}(a)=S_{\varphi} f_a(a)=\frac{1}{\psi(a)}
    \end{equation*}
    and we can subtract it from the both sides of the equivalence.
		Then we proceed as in the proof of Lemma~\ref{lemm:prop_T}. 
    
    As for $\varphi$, we test by $f_a\in\mathcal{M}(\dom, \mu)$ with $f_a^*(s)=\frac{\chi_{(0,a)}(s)}{\varphi(s)}$ for all $a\in(0,R)$. Now, both $T_{\psi} f^{**}$ and $T_{\psi} f$ at $a$ vanish and the validity of
    \begin{equation*}
    	\frac{1}{\varphi(a)}\sup_{0<s<a}\frac{\varphi(s)}{s}\int_0^s\frac{\dd y}{\varphi(y)}=S_{\varphi} f_a^{**}(a)\lesssim S_{\varphi} f_a(a)=\frac{1}{\varphi(a)}
    \end{equation*}
    for any $a\in(0,R)$ gives $\varphi\in B$.
	\end{proof}

\section{Proofs of the main results} \label{sec:proofs}

\begin{proof}[Proof of Theorem~\ref{thm:DKP}]
Let us fix $f\in \mathcal{M}(\dom_1)$ and $t\in(0,R)$. We decompose $f=f^t+f_t$ by
\begin{align*}
	f^t(x) &= \max\bigl\{ |f(x)| - f^*(t),0\bigr\} \sgn f(x),
		\\
	f_t(x) &= \min\bigl\{ |f(x)|, f^*(t)\bigr\} \sgn f(x).
\end{align*}
We then have
\begin{align}
	\bigl(f^t\bigr)^{**}(s)
		& \le \frac{t}{s} f^{**} (t),
		\quad s\in (0,R),
			\label{hh} \\
	\bigl(f_t\bigr)^{**}(s)
		& \le f^{**} (t),
		\quad s\in (0,R).
			\label{dh}
\end{align}
Thus, starting with the quasilinearity of $T$ and the subaditivity of the operator $f\mapsto f^{**}$, we obtain 
\begin{align*} 
		\bigl( Tf \bigr)^{**} (t)
			& \lesssim \bigl( Tf^t + Tf_t \bigr)^{**} (t)
					\\
			& \le \bigl( Tf^t \bigr)^{**} (t) + \bigl( Tf_t \bigr)^{**} (t)
					\\
			& \le \frac{1}{\varphi(t)} \sup_{0<s<R} \varphi(s) \bigl( Tf^t \bigr)^{**} (s)
				+ \frac{1}{\psi(t)} \sup_{0<s<R} \psi(s) \bigl( Tf_t \bigr)^{**} (s)
					\\
			& = \frac{1}{\varphi(t)} \| Tf^t \|_{M_{\varphi}(\dom_2)}
				+ \frac{1}{\psi(t)} \| Tf_t \|_{M_{\psi}(\dom_2)}
					\\
			& \lesssim \frac{1}{\varphi(t)} \| f^t \|_{M_{\varphi}(\dom_1)}
				+ \frac{1}{\psi(t)} \| f_t \|_{M_{\psi}(\dom_1)} \\
            & = \frac{1}{\varphi(t)} \sup_{0<s<R} \varphi(s) ( f^t )^{**} (s)
				+ \frac{1}{\psi(t)} \sup_{0<s<R} \psi(s) ( f_t )^{**} (s)
					\\
			& = \text{I} + \text{II}
                \end{align*}
					due to the boundedness of $T$ from $M_{\varphi}(\dom_1)$ to $M_{\varphi}(\dom_2)$ and from $M_{\psi}(\dom_1)$ to $M_{\psi}(\dom_2)$.
The first term can be estimated, by splitting the supremum and by using \eqref{hh}, as
\begin{align*}
	\text{I}
		& \le \frac{1}{\varphi(t)} \sup_{0<s<t} \varphi(s) f^{**} (s)
				+ \frac{1}{\varphi(t)} \sup_{t<s<R} \varphi(s) ( f^t )^{**} (s)
					\\
		& \le S_{\varphi} f^{**} (t)
				+ f^{**}(t) \frac{t}{\varphi(t)} \sup_{t<s<R} \frac{\varphi(s)}{s}
					\\
		& = S_{\varphi} f^{**} (t)
				+ f^{**}(t)
		\lesssim S_{\varphi} f^{**} (t).
\end{align*}
Similarly, using the estimate \eqref{dh} for the other component of $f$, we get
\begin{align*}
	\text{II}
		& \le \frac{1}{\psi(t)} \sup_{t<s<R} \psi(s) f^{**} (s)
				+ \frac{1}{\psi(t)} \sup_{0<s<t} \psi(s) ( f_t )^{**} (s)
					\\
		& \le T_{\psi} f^{**} (t)
				+ f^{**}(t) \frac{1}{\psi(t)} \sup_{0<s<t} \psi(s)
		\lesssim T_{\psi} f^{**} (t).
\end{align*}
Adding both parts together, we obtain
\begin{equation*}
	( Tf )^{**} (t)
		\lesssim  S_{\varphi} f^{**}  (t)
			+  T_{\psi} f^{**}  (t).
\end{equation*}
Now, thanks to Lemma~\ref{lemm:stars_out}, we can put the double stars away
and continue by
\begin{equation*}
	( Tf )^{**} (t)
		\lesssim S_{\varphi} f (t)
			+ T_{\psi} f (t)
		\lesssim ( S_{\varphi} f + T_{\psi} f )^{**} (t)
\end{equation*}
which is trivially the same as
\begin{equation*}
	( (Tf)^* )^{**} (t)
		\lesssim ( S_{\varphi} f + T_{\psi} f )^{**} (t).
\end{equation*}
Finally, according to Hardy's lemma \cite[Chapter 2, Corollary 4.7]{Ben:88} and properties of r.i.~norms we obtain
\begin{align*}
	\|Tf \|_{X_2(\dom_2)}=\|(Tf)^* \|_{X_2(0,R)}
		&\lesssim \|S_{\varphi} f + T_{\psi} f \|_{X_2(0,R)} \\
        &\le \|S_{\varphi} f \|_{X_2(0,R)}+\|T_{\psi} f \|_{X_2(0,R)}
\end{align*}
and the claim of the theorem follows.
\end{proof}

\begin{remark} \label{rem:discontinuity}
Before we get to the proof of Theorem~\ref{thm:Sgg}, let us first say a few words about additional assumptions \eqref{nontriv} and \eqref{Emb-g-ln} in the case of discontinuous quasiconcave function $\varphi$. 

Since 
\begin{equation*}
S_{\varphi} f(t) = \frac {1} {\varphi(t)} \sup_{0<s<t} \varphi(s) f^*(s)\geq  \frac {\varphi(0{\scriptstyle +})} {\varphi(t)}\|f\|_{L^\infty(\dom)},
\quad t\in(0,R),
\end{equation*}
we get that $S_{\varphi} f(t)=\infty$ on the whole $(0,R)$ for every unbounded $f$. Thus, in the sake of nontriviality, we are only interested in the
situation when $\Gamma^p_{{w}_1}(\dom)\embed  L^\infty(\dom)$. The embeddings of this type
were studied in many papers. By methods of \cite[Remark 2.3]{Car:01}, this
embedding is equivalent to $\sup_{0<t<R}1/\varphi_{\Gamma^p_{{w}_1}}<\infty$
which rewrites as
\begin{equation} \label{sic1}
	\inf_{0<t<R}\biggl(\int_0^t {w}_1(s)\,\dd s
		+ t^p \int_{t}^{R} s^{-p} {w}_1(s)\,\dd s\biggr)>0.
\end{equation}
However, since ${w}_1$ is assumed to be positive,
\eqref{sic1} is equivalent to \eqref{Emb-g-ln}.

Nontriviality also depends on the interplay between quasiconcave function $\varphi$ and the weight ${w}_2$. Indeed, thanks to Lemma \ref{lemm:prop_S},
\begin{align*}
	\| S_\varphi f \|^p_{\Gamma^p_{{w}_2}(0,R)}
		& = \int_0^R \bigl[\bigl( S_\varphi f \bigr)^{**}(s)\bigr]^p {w}_2(s)\,\dd s
			\\
		& \simeq \int_0^R \bigl[\bigl( S_\varphi f \bigr)(s)\bigr]^p {w}_2(s)\,\dd s
        	\\
		& = \int_0^R \varphi^{-p}(s){w}_2(s) \sup_{0<y<s} [f^{*}(y)]^p \varphi^p(y) \,\dd s
			\\
		&\geq \varphi^p(0{\scriptstyle +})\|f\|_{L^\infty(\dom)}^p\int_0^R \varphi^{-p}(s){w}_2(s) \,\dd s
\end{align*}
and \eqref{nontriv} is a necessary assumption in order to
avoid the situation when $S_\varphi f \not\in {\Gamma^p_{{w}_2}(0,R)}$ for any
nontrivial $f$.
\end{remark}

\begin{proof}[Proof of Theorem~\ref{thm:Sgg}]
The necessity follows by plugging the characteristic function of $\mu$-measurable set $E_t$ with $\mu(E_t)=t$ into \eqref{Sgg}. 

As for the sufficiency let us first deal with the case of continuous $\varphi$.
Let us take an arbitrary function $f\in \mathcal{M}(\dom,\mu)$.
We first use the estimates from Lemma \ref{lemm:prop_S} and \eqref{eq:acc}.
We have
\begin{align*}
	\| S_\varphi f \|^p_{\Gamma^p_{{w}_2}(0,R)}
		& = \int_0^R \bigl[\bigl( S_\varphi f \bigr)^{**}(s)\bigr]^p {w}_2(s)\,\dd s
			\\
		& \lesssim \int_0^R \bigl[\bigl( S_\varphi f \bigr)(s)\bigr]^p {w}_2(s)\,\dd s
			\\
		& = \int_0^R \varphi^{-p}(s){w}_2(s) \sup_{0<y<s} [f^{*}(y)]^p \varphi^p(y) \,\dd s
			\\
		& \simeq \int_0^R \varphi^{-p}(s){w}_2(s) \sup_{0<y<s} [f^{*}(y)]^p \int_0^y \varphi^{p-1}(t)\varphi'(t)\,\dd t \,\dd s
			\\
		& \le \int_0^R \varphi^{-p}(s){w}_2(s) \sup_{0<y<s} \int_0^y [f^{*}(t)]^p \varphi^{p-1}(t)\varphi'(t)\,\dd t \,\dd s
			\\
		& = \int_0^R \varphi^{-p}(s){w}_2(s) \int_0^s [f^{*}(t)]^p \varphi^{p-1}(t)\varphi'(t)\,\dd t \,\dd s
			\\
		& = \int_0^R [f^{*}(t)]^p \varphi^{p-1}(t)\varphi'(t) \int_t^R \varphi^{-p}(s){w}_2(s)\,\dd s \,\dd t,
        \end{align*}
where we used that $f^*$ is non-increasing and the Fubini theorem. Thanks to this estimate, we only need that
\begin{equation} \label{sic2}
	\int_{0}^{R} [f^*(t)]^p w(t)\,\dd t
		\lesssim \int_{0}^{R} [f^{**}(t)]^p {w}_1(t)\,\dd t
\end{equation}
where
\begin{equation} \label{eq:wdef}
	w(t)=p\varphi^{p-1}(t)\varphi'(t) \int_t^R \varphi^{-p}(s){w}_2(s)\,\dd s,
		\quad t\in(0,R).
\end{equation}
By \cite[Theorem 3.2]{Neu:91}, the inequality \eqref{sic2} holds if and only if
\begin{equation} \label{eq:Neucond}
	\int_0^t w(s)\,\dd s 
		\lesssim
	\int_0^t {w}_1(s)\,\dd s
		+ t^p \int_{t}^{R} s^{-p} {w}_1(s)\,\dd s,
		\quad t\in(0,R),
\end{equation}
which is equivalent to \eqref{Sggpp} by integration by parts.

For the sufficiency in the case $\varphi$ is discontinuous, we start similarly 
\begin{align*}
	\| S_\varphi f \|^p_{\Gamma^p_{{w}_2}(0,R)}
		& \lesssim \int_0^R \varphi^{-p}(s){w}_2(s) \sup_{0<y<s} [f^{*}(y)]^p \left(p\int_0^y \varphi^{p-1}(t)\varphi'(t)\,\dd t+\varphi^p(0{\scriptstyle +})\right) \dd s
        	\\
		& \simeq \int_0^R \varphi^{-p}(s){w}_2(s) \sup_{0<y<s} [f^{*}(y)]^p \int_0^y \varphi^{p-1}(t)\varphi'(t)\,\dd t \,\dd s 
        	\\
        & \quad + \|f\|_{L^\infty(\dom)}^p\int_0^R \varphi^{-p}(s){w}_2(s) \,\dd s.
\end{align*}
The second term is estimated by a constant multiple of
$\|f\|^p_{\Gamma_{{w}_1}^p(\dom)}$, thanks to the assumptions. As for the first one,
we proceed in the same way as above and again, due to \cite[Theorem
3.2]{Neu:91}, we obtain the sufficiency of \eqref{eq:Neucond}
where $w$ is defined as in \eqref{eq:wdef}. Now, by integration by parts of the
left hand side, we get
\begin{align*}
	\int_0^t w(s)\,\dd s 
		& =
			\int_0^t {w}_2(s)\,\dd s
			+ \varphi^p(t) \int_t^R \varphi^{-p}(s) {w}_2(s) \,\dd s
			\\
		& \quad - \varphi^p(0{\scriptstyle +})\int_0^R \varphi^{-p}(s){w}_2(s) \,\dd s,
			\quad t\in (0,R),
\end{align*}
and clearly \eqref{Sggpp} is also sufficient for \eqref{Sgg} in this case.
\end{proof}

\begin{proof}[Proof of Theorem~\ref{thm:Tgg}]
Assume that \eqref{Tggpp} holds. We have, according to~Lemma~\ref{lemm:prop_T},
\begin{align*}
	\| T_\psi f \|^p_{\Gamma^p_{{w}_2}(0,R)}
		& = \int_0^R \bigl[\bigl( T_\psi f \bigr)^{**}(s)\bigr]^p {w}_2(s)\,\dd s
			\\
		& \lesssim \int_0^R \bigl[ T_\psi f (s)\bigr]^p {w}_2(s)\,\dd s
			+ \int_0^R [ f^{**}(s) ]^p {w}_2(s)\,\dd s
			\\
		& = \text{I} + \text{II}.
\end{align*}
Next,
\begin{align*}
	\text{I}
		& \lesssim \int_0^R \psi^{-p}(s) {w}_2(s) \sup_{s<y<R}
				\bigl( \psi^p(y) - \psi^p(s) \bigr) [f^*(y)]^p\,\dd s
				\\
			& \quad + \int_0^R \psi^{-p}(s) {w}_2(s) \sup_{s<y<R} \psi^p(s) [f^*(y)]^p\,\dd s
			\\
		& = p \int_0^R \psi^{-p}(s) {w}_2(s) \sup_{s<y<R} [f^*(y)]^p
				\int_{s}^{y} \psi^{p-1}(t)\psi'(t) \,\dd t\, \dd s
				\\
			& \quad + \int_0^R {w}_2(s) \sup_{s<y<R} [f^*(y)]^p\,\dd s
			\\
		& \le p \int_0^R \psi^{-p}(s) {w}_2(s)
				\int_{s}^{R} [f^*(t)]^p \psi^{p-1}(t)\psi'(t) \,\dd t\, \dd s
				\\
			& \quad + \int_0^R [f^*(s)]^p {w}_2(s) \,\dd s
			\\
		& \le p\int_{0}^{R} [f^*(t)]^p \psi^{p-1}(t)\psi'(t) 
				\int_0^t \psi^{-p}(s) {w}_2(s)\,\dd s\, \dd t
				\\
			& \quad + \int_0^R [f^*(s)]^p {w}_2(s) \,\dd s
			\\
		& = \int_{0}^{R} [f^*(t)]^p w(t)\,\dd t,
\end{align*}
where we set
\begin{equation*}
	w(t) = {w}_2(t) + p\psi^{p-1}(t)\psi'(t)
		\int_0^t \psi^{-p}(s) {w}_2(s)\,\dd s,
			\quad t\in(0,R).
\end{equation*}
Now, it suffices to show that \eqref{Tggpp} implies
\begin{equation} \label{sic3}
		\int_{0}^{R} [f^{**}(t)]^p {w}_2(t)\,\dd t
		\lesssim 	\int_{0}^{R} [f^{**}(t)]^p {w}_1\,\dd t
\end{equation}
and also
\begin{equation} \label{sic4}
		\int_{0}^{R} [f^*(t)]^p w(t)\,\dd t
		\lesssim \int_{0}^{R} [f^{**}(t)]^p {w}_1\,\dd t.
\end{equation}
The embedding \eqref{sic3} holds if and only if
\begin{equation}
	\int_{0}^{t} {w}_2(s)\,\dd s + t^p\int_{t}^R s^{-p}{w}_2(s)\,\dd s
		\lesssim \int_0^t {w}_1(s)\,\dd s
			+ t^p \int_t^R s^{-p}{w}_1(s)\,\dd s,
			\quad t\in(0,R),
	\label{GG}
\end{equation}
due to \cite[Theorem 3.2]{Gol:96}, while \eqref{sic4} is, by \cite[Theorem 3.2]{Neu:91}, equivalent to 
\begin{equation*}
	\int_{0}^{t} w(s)\,\dd s
		\lesssim \int_0^t {w}_1(s)\,\dd s
			+ t^p \int_t^R s^{-p}{w}_1(s)\,\dd s,
			\quad t\in(0,R),
\end{equation*}
which is the same as
\begin{equation}
	\psi^p(t) \int_{0}^{t} \psi^{-p}(s) {w}_2(s)\,\dd s
		\lesssim \int_0^t {w}_1(s)\,\dd s
			+ t^p \int_t^R s^{-p}{w}_1(s)\,\dd s,
			\quad t\in(0,R),
	\label{GL}
\end{equation}
by integration by parts.
Finally, since
\begin{equation*}
	\int_{0}^{t} {w}_2(s)\,\dd s
		\le \psi^p(t) \int_{0}^{t} \psi^{-p}(s) {w}_2(s)\,\dd s,
		\quad t\in (0,R),
\end{equation*}
due to the fact that $\psi$ is increasing, \eqref{Tggpp} ensures both \eqref{GG} and \eqref{GL}.

The necessity follows again by evaluating both sides of \eqref{Tgg} on characteristics functions.
\end{proof}
    
\begin{proof}[Proof of Theorem~\ref{thm:A}]
Let us first show that the validity of both conditions for the boundedness of
$S_\varphi$ and $T_\psi$ on Lorentz gamma spaces \eqref{Sggpp} and
\eqref{Tggpp} is equivalent to the condition \eqref{STggpp}. Indeed, since
$\psi(s)$ and $s/\varphi(s)$ are both increasing, we have
\begin{equation*}
	\int_0^t {w}_2(s)\,\dd s
	\leq
		\psi^p(t) \int_0^t \psi^{-p}(s) {w}_2(s) \,\dd s
\end{equation*}
and
\begin{equation*}
	t^p \int_t^R s^{-p}{w}_2(s)\,\dd s
	\leq
		\varphi^p(t) \int_t^R \varphi^{-p}(s) {w}_2(s) \,\dd s.
\end{equation*}
Our result then follows from  Theorem~\ref{thm:Tgg} and Theorem~\ref{thm:Sgg} used together with Theorem~\ref{thm:DKP}.
\end{proof}

\section*{acknowledgement}

We would like to express our thanks to Lubo\v s Pick and Martin K\v repela for
valuable comments and suggestions. We would also like to thank the referees for
their careful reading of the manuscript and for pointing out weak spots that
helped us to improve the final version of the paper. 

\section*{Funding}

This research was partly supported by the grant P201-13-14743S of the Grant
Agency of the Czech Republic and by the grant SVV-2016-260335 of the Charles
University.

\bibliography{arxiv_update}

\def\cprime{$'$}
\begin{thebibliography}{10}

\bibitem{Ben:88}
C.~Bennett and R.~Sharpley.
\newblock {\em Interpolation of operators}, volume 129 of {\em Pure and Applied
  Mathematics}.
\newblock Academic Press, Inc., Boston, MA, 1988.

\bibitem{Car:01}
M.~Carro, L.~Pick, J.~Soria, and V.~D. Stepanov.
\newblock On embeddings between classical {L}orentz spaces.
\newblock {\em Math. Inequal. Appl.}, 4(3):397--428, 2001.

\bibitem{Dmi:78}
V.~I. Dmitriev and S.~G. Kre{\u\i}n.
\newblock Interpolation of operators of weak type.
\newblock {\em Anal. Math.}, 4(2):83--99, 1978.

\bibitem{Gol:96}
M.~L. Gol{\cprime}dman, H.~P. Heinig, and V.~D. Stepanov.
\newblock On the principle of duality in {L}orentz spaces.
\newblock {\em Canad. J. Math.}, 48(5):959--979, 1996.

\bibitem{Ker:14}
R.~Kerman, C.~Phipps, and L.~Pick.
\newblock Marcinkiewicz interpolation theorems for {O}rlicz and {L}orentz gamma
  spaces.
\newblock {\em Publ. Mat.}, 58(1):3--30, 2014.

\bibitem{Ker:06}
R.~Kerman and L.~Pick.
\newblock Optimal {S}obolev imbeddings.
\newblock {\em Forum Math.}, 18(4):535--570, 2006.

\bibitem{Ker:09}
R.~Kerman and L.~Pick.
\newblock Optimal {S}obolev imbedding spaces.
\newblock {\em Studia Math.}, 192(3):195--217, 2009.

\bibitem{Kre:82}
S.~G. Kre{\u{\i}}n, Y.~{\={I}}. Petun{\={\i}}n, and E.~M. Sem{\"{e}}nov.
\newblock {\em Interpolation of linear operators}, volume~54 of {\em
  Translations of Mathematical Monographs}.
\newblock American Mathematical Society, Providence, R.I., 1982.
\newblock Translated from the Russian by J. Sz\H ucs.

\bibitem{Kuf:07}
A.~Kufner, L.~Maligranda, and L.-E. Persson.
\newblock {\em The {H}ardy inequality}.
\newblock Vydavatelsk{\'y} Servis, Plze\v n, 2007.
\newblock About its history and some related results.

\bibitem{Neu:91}
C.~J. Neugebauer.
\newblock Weighted norm inequalities for averaging operators of monotone
  functions.
\newblock {\em Publ. Mat.}, 35(2):429--447, 1991.

\bibitem{Olh:11}
R.~O{\soft{l}}hava.
\newblock Optimal pairs of function spaces for weighted hardy operators.
\newblock Master's thesis, Faculty of Mathematics and Physics at Charles
  University, 2011.

\bibitem{ONe:68}
R.~O'Neil.
\newblock Integral transforms and tensor products on {O}rlicz spaces and
  {$L(p,\,q)$} spaces.
\newblock {\em J. Analyse Math.}, 21:1--276, 1968.

\bibitem{Str:79}
J.~O. Str{\"o}mberg.
\newblock Bounded mean oscillation with {O}rlicz norms and duality of {H}ardy
  spaces.
\newblock {\em Indiana Univ. Math. J.}, 28(3):511--544, 1979.

\end{thebibliography}

\end{document}